\documentclass[11pt]{article}
\usepackage[T1]{fontenc}
\usepackage{lmodern}
\usepackage{microtype}
\usepackage{amsmath,amssymb,mathtools,amsthm}
\usepackage{tikz}
\usepackage[colorlinks=true,linkcolor=black,urlcolor=blue]{hyperref}
\usepackage[capitalise]{cleveref}
\newtheorem{proposition}{Proposition}[section]
\theoremstyle{definition}
\newtheorem{example}{Example}[section]
\newtheorem*{remark}{Remark}

\newcommand{\bbR}{\mathbb{R}}
\newcommand{\Reduction}{\mathcal{R}}

\title{Networks with Memory and Isospectral Reduction}
\author{Grant Molnar}
\date{}

\begin{document}
\maketitle

\begin{abstract}
Restricting the linear dynamics of a graph to selected vertices generally introduces memory, because interactions can pass through hidden vertices before returning. The resulting rational memory kernels give isospectral reduction a dynamical interpretation and identify latent graph symmetries with symmetries of the observed evolution.
\end{abstract}

\section{Introduction}

We consider a finite undirected graph \(G\) with vertex set \(V(G)\) and real edge weights. Its adjacency matrix is the symmetric matrix \(A \in \bbR^{V(G) \times V(G)}\), where \(A_{uv}\) is the weight of the edge joining \(u,v \in V(G)\), or zero if the edge is absent. We evolve states \(z_n \in \bbR^{V(G)}\) by \(z_{n + 1} \coloneqq A z_n\). Starting from the indicator vector \(z_0 = e_v\) at \(v\), we have \((z_n)_u = (A^n)_{uv}\). This entry sums the products of edge weights along all walks of length \(n\) from \(v\) to \(u\). Thus the powers of \(A\) describe how states propagate through the graph.

\begin{example}\label{ex:path}
Suppose we have an undirected path \(a-c-b\) with unit edge weights, and suppose moreover that we may only observe the endpoints \(a\) and \(b\). The induced subgraph \(\{a, b\}\) has no edges. We write \(x_n \in \bbR^2\) for the states at \(a\) and \(b\), in that order, and write \(y_n \in \bbR\) for the state at \(c\). The original evolution satisfies
\[
 x_{n + 1} = \begin{pmatrix}1 \\ 1\end{pmatrix} y_n
 \qquad \text{and} \qquad
 y_{n + 1} = \begin{pmatrix}1 & 1\end{pmatrix} x_n.
\]
When \(y_0 = 0\), we have
\[
 x_1 = 0
 \qquad \text{and} \qquad
 x_2 = \begin{pmatrix}1 & 1 \\ 1 & 1\end{pmatrix} x_0.
\]
No matrix \(F\) acting on the endpoint coordinates can reproduce both steps for every \(x_0\). The first would require \(F = 0\), while the second is generally nonzero.

Eliminating \(y_n\), we obtain the exact endpoint recurrence
\[
 x_{n + 1} = \begin{pmatrix}1 & 1 \\ 1 & 1\end{pmatrix} x_{n - 1}
 \qquad \text{for} \ n \ge 1.
\]
The middle vertex has disappeared, but its influence remains as a delayed interaction.
\end{example}

In this note, we consider networks with hidden vertices more generally. In larger networks, influence can spend many steps among hidden vertices before returning, and cycles in the hidden subgraph can produce arbitrarily long delays. The observed evolution may therefore depend on past states even when the original network is memoryless.

\section{Networks with interfaces}\label{sec:networks-with-interfaces}

In this section, we generalize \Cref{ex:path} and study the resulting model.

We consider a finite-dimensional real vector space \(Z = X \oplus Y\), a linear transformation \(T: Z \to Z\), and an initial state \(z_0 \in Z\). We observe only the \(X\)-component of the autonomous evolution \(z_{n + 1} = T z_n\). Writing \(z_n = x_n + y_n\) with \(x_n \in X\) and \(y_n \in Y\), and decomposing \(T\) accordingly, we obtain
\begin{equation}\label{eqn:fundamental-decomposition}
 T = \begin{pmatrix}M & B \\ C & D\end{pmatrix} \qquad \text{and} \qquad
 \begin{aligned}
 x_{n + 1} &= M x_n + B y_n \\
 y_{n + 1} &= C x_n + D y_n.
 \end{aligned}
\end{equation}
Iterating the second equation of \eqref{eqn:fundamental-decomposition} gives
\[
 y_n = D^n y_0 + \sum_{j = 0}^{n - 1} D^{n - 1 - j}C x_j,
\]
so the observed coordinate \(x_n\) satisfies
\begin{equation}\label{eq:memory}
 x_{n + 1} = M x_n + \sum_{j = 0}^{n - 1}BD^{n - 1 - j}C x_j + BD^n y_0.
\end{equation}

The final term in \eqref{eq:memory} records the contribution of the hidden initial state. With \(y_0 = 0\), we define
\[
 K_0 \coloneqq M \qquad \text{and} \qquad
 K_j \coloneqq BD^{j - 1}C \qquad \text{for } j \ge 1.
\]
Then \eqref{eq:memory} becomes
\[
 x_{n + 1} = \sum_{j = 0}^{n} K_j x_{n - j}.
\]
The sequence \((K_j)_{j \ge 0}\) is the \emph{memory kernel} of the observed dynamics on \(X\). Its \(j\)th coefficient describes an interaction delayed by \(j\) steps, so memoryless systems have \(K_j = 0\) for \(j \ge 1\).

We extend sequences by zero to negative indices and use \(\zeta^{-1}\) to denote one step of delay, so that \(\zeta^{-1}x_n \coloneqq x_{n - 1}\), and likewise for \(y_n\) and \(z_n\). The memory kernel has generating function
\begin{equation}\label{eq:kernel}
\begin{aligned}
 K(\zeta) &\coloneqq \sum_{j \ge 0} K_j \zeta^{-j}\\
 &= M + B(\zeta I - D)^{-1}C.
\end{aligned}
\end{equation}
The second expression follows from the formal expansion \((\zeta I - D)^{-1} = \sum_{j \ge 0} D^j \zeta^{-j - 1}\). Thus eliminating finitely many internal coordinates naturally produces rational memory kernels.

We write \(i: X \hookrightarrow Z\) and \(p: Z \to X\) for the natural inclusion and projection maps. The operators
\begin{equation}\label{eqn:interface-evolution}
 G_n \coloneqq p T^n i \qquad \text{for } n \ge 0
\end{equation}
describe the evolution seen at the interface \(X\) when the internal state \(y_0\) is initially zero. Unlike the powers of \(T\), they need not satisfy \(G_{m + n} = G_m G_n\), because a trajectory may leave the interface \(X\) and return again later. The generating function for \eqref{eqn:interface-evolution} is
\begin{equation}\label{eq:response}
\begin{aligned}
 H_X(\zeta) &\coloneqq \sum_{n \ge 0} G_n \zeta^{-n - 1}\\
 &= p(\zeta I - T)^{-1}i = (\zeta I - K(\zeta))^{-1}.
\end{aligned}
\end{equation}
For the last equality, we solve \((\zeta I - T)(x + y) = x^\prime\) with \(x^\prime \in X\). The hidden equation in \eqref{eqn:fundamental-decomposition} gives \(y = (\zeta I - D)^{-1}Cx\), and the visible equation in \eqref{eqn:fundamental-decomposition} becomes \((\zeta I - K(\zeta))x = x^\prime\). Thus the compressed evolution determines the memory kernel through
\begin{equation}\label{eq:inverse}
 K(\zeta) = \zeta I - H_X(\zeta)^{-1}.
\end{equation}
In particular, the response describes evolution from states initially supported on the interface, not all observable trajectories arising from arbitrary hidden initial states. Conversely, every proper rational matrix kernel (one bounded at \(\zeta = \infty\)) has a finite-dimensional realization of the form \eqref{eq:kernel}, as in classical linear systems theory \cite{PW1998}. Thus finite-dimensional elimination gives precisely the proper rational kernels; a general memory kernel need not be rational.

We might observe still fewer coordinates. We fix a basis \((e_i)_{i \in \mathcal I}\) of \(X\), indexed by a finite set \(\mathcal I\), and retain only the coordinates indexed by \(S \subseteq \mathcal I\). We write \(W \coloneqq \mathcal I \setminus S\). For subsets \(A, B \subseteq \mathcal I\) and a matrix \(F\) indexed by \(\mathcal I\), we write \(F_{AB}\) for the submatrix with rows indexed by \(A\) and columns by \(B\), and \(I_S\) for the identity on the coordinates in \(S\). The observed evolution is described by \((G_n)_{SS}\) and therefore has response \(H_{SS}(\zeta)\), where \(H(\zeta) \coloneqq (\zeta I - K(\zeta))^{-1}\). By \eqref{eq:inverse}, its memory kernel is
\begin{equation}\label{eq:general-reduction}
\begin{aligned}
 \Reduction_S(K;\zeta)
 &\coloneqq \zeta I_S - H_{SS}(\zeta)^{-1}\\
 &= K_{SS}(\zeta) + K_{SW}(\zeta)(\zeta I_W - K_{WW}(\zeta))^{-1}K_{WS}(\zeta).
\end{aligned}
\end{equation}
The second expression follows by block elimination. The reduced kernel is again proper and rational. If \(U \subseteq S \subseteq \mathcal I\), observing \(U\) directly or through \(S\) gives the same response \(H_{UU}(\zeta)\), and therefore
\begin{equation}\label{eq:successive}
 \Reduction_U(\Reduction_S(K;\zeta);\zeta) = \Reduction_U(K;\zeta).
\end{equation}

\section{Isospectral reduction}

We now return to a finite undirected graph \(G\) with real edge weights and adjacency matrix \(A\). We retain a set \(S \subseteq V(G)\) of observed vertices and write \(W \coloneqq V(G) \setminus S\), giving the partition
\[
 A = \begin{pmatrix}A_{SS} & A_{SW} \\ A_{WS} & A_{WW}\end{pmatrix}.
\]

The general reduction \eqref{eq:general-reduction}, applied to the constant kernel \(K(\zeta) = A\), gives
\begin{equation}\label{eq:reduction}
 \Reduction_S(A;\zeta)
 = A_{SS} + A_{SW}(\zeta I - A_{WW})^{-1}A_{WS},
\end{equation}
which is the Bunimovich--Webb isospectral reduction \cite[Ch.~2]{BW2014}. The corresponding response is the restricted resolvent \cite[\S\S~2,~4]{KT2024}
\begin{equation}\label{eq:resolvent}
 (\zeta I - \Reduction_S(A;\zeta))^{-1}
 = \bigl((\zeta I - A)^{-1}\bigr)_{SS}.
\end{equation}

\begin{example}\label{ex:path-reduction}
Returning to \Cref{ex:path}, we order the vertices as \(a,b,c\) and retain \(S = \{a,b\}\). We obtain
\[
 A = \begin{pmatrix}0 & 0 & 1 \\ 0 & 0 & 1 \\ 1 & 1 & 0\end{pmatrix}
 \qquad \text{and} \qquad
 \Reduction_S(A;\zeta) = \frac{1}{\zeta}
 \begin{pmatrix}1 & 1 \\ 1 & 1\end{pmatrix}.
\]
In particular, the diagonal weights record excursions from an endpoint through the middle vertex and back, rather than loops in the original graph.
\end{example}

For \(j \ge 1\), the coefficient of \(\zeta^{-j}\) in \eqref{eq:reduction} is \(A_{SW}A_{WW}^{j - 1}A_{WS}\). Its entries sum the weights of walks of length \(j + 1\) between interface vertices whose intermediate vertices lie in \(W\) \cite[\S~3.2]{KSSW2020}.

We consider an eigenvector \(z \ne 0\) with \(A z = \lambda z\), where \(\lambda \ne 0\). The sequence \(z_n = \lambda^n z\), indexed over \(\mathbb Z\), is a trajectory of the system, which we call an \emph{exponential mode}. Delaying it by \(j\) steps multiplies it by \(\lambda^{-j}\). Thus the delay kernel is evaluated at \(\zeta = \lambda\). We decompose \(z = x + y\) into its observed and hidden coordinates. When \(\lambda I - A_{WW}\) is invertible, eliminating \(y\) from \(A z = \lambda z\) gives
\[
 \Reduction_S(A;\lambda)x = \lambda x.
\]
The Schur-complement identity
\begin{equation}\label{eq:determinant}
 \det(\zeta I - A) = \det(\zeta I - A_{WW})
                 \det(\zeta I - \Reduction_S(A;\zeta))
\end{equation}
recovers the full characteristic polynomial, including factors that may cancel against poles of the reduced determinant.

\begin{remark}
The constructions above do not require symmetry or real scalars. The memory-kernel formulas and the reduction and determinant identities hold for finite weighted directed graphs over any field \(k\), with rational expressions interpreted over \(k(\zeta)\) and their generating series over \(k((\zeta^{-1}))\). Of course, computing eigenvalues may require an extension of \(k\).

The Perron-centrality consequence below requires nonnegative real weights and a connected graph.
\end{remark}

\section{Latent symmetry}

The reduced graph can have symmetries absent from the original graph.

\begin{example}\label{ex:latent}
We consider an undirected weighted graph with vertices \(a,b,u,v,w\). We join \(a\) to \(b\) with weight \(1\), \(a\) to \(u\) with weight \(5\), and \(b\) to \(v\) and \(w\) with weights \(3\) and \(4\), respectively.
\begin{center}
\begin{tikzpicture}[scale=0.95, every node/.style={font=\small}, vertex/.style={circle,draw,minimum size=6mm,inner sep=0pt}]
 \node[vertex,thick] (a) at (-3.25,0) {\(a\)};
 \node[vertex,thick] (b) at (-1.05,0) {\(b\)};
 \node[vertex] (u) at (-4.28,0.9) {\(u\)};
 \node[vertex] (v) at (-0.03,0.9) {\(v\)};
 \node[vertex] (w) at (-0.03,-0.9) {\(w\)};
 \draw (a) -- (b);
 \draw (a) -- (u);
 \draw (b) -- (v);
 \draw (b) -- (w);
 \node at (-2.15,0.32) {\(1\)};
 \node at (-4.02,0.35) {\(5\)};
 \node at (-0.65,0.72) {\(3\)};
 \node at (-0.65,-0.72) {\(4\)};
 \draw[->] (0.42,0) -- (1.08,0);
 \node[vertex,thick] (ra) at (1.95,0) {\(a\)};
 \node[vertex,thick] (rb) at (4.15,0) {\(b\)};
 \draw (ra) -- (rb);
 \draw (ra) edge[loop above,min distance=9mm,looseness=7] (ra);
 \draw (rb) edge[loop above,min distance=9mm,looseness=7] (rb);
 \node at (3.05,0.32) {\(1\)};
 \node at (1.95,1.38) {\(25/\zeta\)};
 \node at (4.15,1.38) {\(25/\zeta\)};
\end{tikzpicture}
\end{center}
We retain \(S = \{a,b\}\) and obtain
\[
 \Reduction_S(A;\zeta)
 = \begin{pmatrix}25/\zeta & 1 \\ 1 & 25/\zeta\end{pmatrix},
\]
The edge of weight \(1\) retains the direct interaction between \(a\) and \(b\). Each loop of weight \(25/\zeta\) represents a one-step delay: its coefficient is \(5^2\) at \(a\) and \(3^2 + 4^2\) at \(b\), from two-step walks through the eliminated vertices. The reduced graph admits the transposition of \(a\) and \(b\), whereas the original graph does not, since \(a\) has two neighbors and \(b\) has three.
\end{example}

An automorphism of the reduced graph is a permutation of \(S\) commuting with its rational adjacency matrix. The compressed evolution provides an equivalent description.

\begin{proposition}\label{prop:symmetry}
We write \(H_S(\zeta) \coloneqq ((\zeta I - A)^{-1})_{SS}\). For any permutation matrix \(Q\) on \(S\), the following conditions are equivalent.
\[
 \begin{aligned}
 \text{(i)}\quad &Q \Reduction_S(A;\zeta) = \Reduction_S(A;\zeta)Q \\
 \text{(ii)}\quad &Q H_S(\zeta) = H_S(\zeta) Q \\
 \text{(iii)}\quad &Q(A^n)_{SS} = (A^n)_{SS} Q
 \qquad \text{for } n \ge 0.
 \end{aligned}
\]
\end{proposition}
\begin{proof}
The first two conditions are equivalent by \eqref{eq:resolvent}. The third is the coefficientwise form of the second, by \eqref{eq:response}.
\end{proof}

Thus a latent automorphism preserves the entire interface response, even if it does not extend to an automorphism of the original graph \cite{SW2019}.

For symmetric \(A\) and \(S = \{a, b\}\), transposition of \(a\) and \(b\) commutes with \(H_S(\zeta)\) exactly when
\[
 (A^n)_{aa} = (A^n)_{bb} \qquad \text{for } n \ge 0.
\]
These equalities of closed-walk counts characterize cospectral vertices.

Response symmetries also constrain Perron eigenvector centrality. We assume \(A\) is real, nonnegative, and irreducible, and write \(\rho\) for its Perron eigenvalue and \(r,s\) for positive right and left eigenvectors. The interface resolvent has residue
\[
 \operatorname*{Res}_{\zeta = \rho}H_S(\zeta)
 = \frac{r_S s_S^{\mathsf T}}{s^{\mathsf T} r}.
\]
If \(Q\) commutes with \(H_S(\zeta)\), it preserves the image of this residue. Hence \(Q r_S = c r_S\) for some \(c > 0\). Since \(Q\) has finite order, \(c = 1\), and vertices exchanged by \(Q\) have equal Perron eigenvector centrality. In \Cref{ex:latent}, \(a\) and \(b\) consequently have equal Perron centrality despite their different degrees.

\section{Further reading}

Bunimovich and Webb's monograph is the foremost exposition of isospectral reduction \cite{BW2014}. Its branch formula expresses rational weights as sums over paths through eliminated vertices \cite[Ch.~2]{BW2014}. Kempton, Sinkovic, Smith, and Webb \cite[\S~3.2]{KSSW2020} give the walk-generating interpretation, while Kempton and Tolbert \cite{KT2024} develop restricted resolvents and their connection with quantum walks. Smith and Webb \cite{SW2019} study latent symmetry and eigenvector centrality, and the connection with cospectral vertices is developed in \cite{KSSW2020}.

Our memory calculation in \Cref{sec:networks-with-interfaces} arises as the elimination of hidden states in linear systems theory; see Polderman and Willems \cite{PW1998} on realizations. The Mori--Zwanzig formalism, derived for linear systems by Gouasmi, Parish, and Duraisamy \cite[\S~2]{GPD2017}, likewise describes unresolved variables by a history-dependent term. Bunimovich and Webb also study delayed interactions induced by restricting dynamical networks \cite[Ch.~3]{BW2014} \cite{BW2013}.

\end{document}